\documentclass[]{article}

\usepackage{geometry}
\usepackage{booktabs}
\usepackage{graphicx}
\usepackage{placeins}

\usepackage{mathtools} 
\usepackage{mathrsfs}
\usepackage{amsmath}
\usepackage{amsfonts}
\usepackage{amssymb}
\usepackage{xcolor}

\usepackage{amsthm}
\theoremstyle{definition}

\newtheorem{thm}{Theorem}[section]
\newtheorem{prop}[thm]{Proposition}

\newtheorem{lemma}[thm]{Lemma}

\newtheorem{defn}[thm]{Definition}
\newtheorem{rmk}[thm]{Remark}

\usepackage{csquotes}
\usepackage[
    backend=biber,
    style=numeric,
    sorting=none,
    doi=true,
    url=false,
    isbn=false
]{biblatex}
\usepackage{hyperref}
\hypersetup{hidelinks}

\newcommand{\diag}{\text{diag}}
\newcommand{\diam}{\text{diam}}

\newcommand{\norm}[1]{\left\lVert#1\right\rVert}

\title{Inverse Wave Problems on Graphs}
\author{Amir Moradifam\footnote{Department of Mathematics, University of California, Riverside, California, USA. E-mail: amirm@ucr.edu} \and Oliver P. Waite\footnote{Department of Mathematics, University of California, Riverside, California, USA. E-mail: oliver.waite@email.ucr.edu}}
\date{}

\begin{document}

\maketitle

\begin{abstract}
    We study inverse wave problems on finite graphs from time-dependent measurements on a prescribed set of vertices. With the graph geometry known, we consider recovery of either the initial signal or the vertex wave speeds when the other is given. Initial-signal recovery is characterized by observability of the associated finite-dimensional system, while a single experiment determines the vertex wave speeds under a zero forcing condition and a natural nonharmonicity assumption. We establish stability estimates for both recovery problems, develop a nonlinear least-squares reconstruction method, and extend the results to positively weighted graphs. Numerical simulations illustrate the reconstruction and examine the effects of observation geometry, time horizon, coefficient localization, and measurement noise on its performance. Our results are motivated both by numerical approaches to continuous inverse wave problems, where spatial discretization leads naturally to graph-based systems, and by applications in which the underlying dynamics are inherently discrete and evolve on networks.
\end{abstract}

\section{Introduction}

Inverse wave problems arise when the evolution of a wave is observed on all or part of the boundary of a region and one seeks to recover the initial source, the wave speed, or both. In the continuous setting, a representative model is
\begin{equation}
\left\{
\begin{aligned}
u_{tt}(x,t) &= c(x)^2\Delta u(x,t),
    && (x,t)\in\mathbb{R}^d\times(0,T),\\
u(x,0) &= f(x),
    && x\in\mathbb{R}^d,\\
u_t(x,0) &= 0,
    && x\in\mathbb{R}^d.
\end{aligned}
\right.
\label{eq:wave-equation}
\end{equation}
Here $c(x)>0$ is the wave speed and $f$ is the initial source. We assume that $f$ is supported in a bounded region $\Omega\subset\mathbb{R}^d$ and that $c$ is known outside $\Omega$. For a fixed observation time $T>0$, the full boundary measurement is
$$
\Lambda(f,c):=u|_{(0,T)\times\partial\Omega}.
$$
More generally, if $\Sigma\subset\partial\Omega$ is the observed portion of the boundary, one may consider the partial measurement
$$
\Lambda_{\Sigma}(f,c):=u|_{(0,T)\times\Sigma}.
$$
The corresponding inverse problems are to recover $f$ when $c$ is known, to recover $c$ when $f$ is known, or to recover both $f$ and $c$ simultaneously from these measurements.

There is a substantial literature on inverse wave problems of this type. When the wave speed is known, recovery of the initial source is a classical inverse source problem and is closely related to observability; see, for example, \cite{StefanovUhlmann}. Related boundary Cauchy reconstruction problems are studied in \cite{Demchenko}. When the initial source is known, Stefanov and Uhlmann proved uniqueness and stability for recovery of the wave speed from one measurement under geometric and nondegeneracy assumptions \cite{StefanovUhlmann}. A Carleman-based numerical reconstruction method for a related single boundary-flux measurement problem was developed in \cite{BaudouinDeBuhanErvedozaOsses}. In contrast, simultaneous recovery of $f$ and $c$ from the same measurements remains open in general. Representative positive results under additional structural assumptions include \cite{FinchHickmann,LiuUhlmann,KnoxMoradifam,KianUhlmann,KianLiu,Moradifam}. A one-dimensional passive-measurement result appears in \cite{Feizmohammadi}. Moreover, the linearized simultaneous recovery problem is unstable in every scale of Sobolev spaces \cite{StefanovUhlmannInstability}.

To solve these continuous inverse problems numerically, one typically replaces the spatial domain by a finite mesh. Writing the wave equation in the equivalent mass form
$$
m(x)u_{tt}(x,t)-\Delta u(x,t)=0,
\qquad
m(x)=c(x)^{-2},
$$
let $\Gamma_h=(V_h,E_h)$ be a mesh graph with nodes $x_1^h,\ldots,x_{N_h}^h$. Let $U_h(t)\in\mathbb{R}^{N_h}$ denote the vector of discrete nodal unknowns, with
$$
\bigl(U_h(t)\bigr)_i\approx u(x_i^h,t).
$$
A finite-difference or mass-lumped finite-element discretization then leads to a system of the form
$$
M_hU_{h,tt}(t)+L_hU_h(t)=0,
$$
where $M_h$ is a diagonal discrete mass matrix and $L_h$ is the discrete spatial operator. On a uniform finite-difference mesh, for example,
$$
M_h=\operatorname{diag}\bigl(m(x_1^h),\ldots,m(x_{N_h}^h)\bigr),
\qquad
L_h=h^{-2}L_{\Gamma_h},
$$
away from the outer boundary of the computational mesh. Thus inverse wave problems on graphs arise naturally in numerical discretizations of continuous inverse wave problems.

The same finite-dimensional wave systems also arise in models whose underlying structure is intrinsically discrete or networked. Wave and hyperbolic PDE dynamics on graphs have been used to design or reinterpret message-passing architectures in graph neural networks \cite{BhaskarEtAl,YueEtAl2025a,YueEtAl2025b}. Motivated by both settings, let $\Gamma=(V,E)$ be a graph with $N$ vertices and graph Laplacian $L$. We consider
\begin{equation}
\left\{
\begin{aligned}
X_{tt}(t) &= -C^2LX(t),
    && t\geq0,\\
X(0) &= f,\\
X_t(0) &= 0,
\end{aligned}
\right.
\label{eq:intro-graph-wave}
\end{equation}
where $f\in\mathbb{R}^N$ is the initial signal and
$$
C=\operatorname{diag}(c_1,\ldots,c_N)
$$
contains the vertex wave speeds.

A direct physical realization is provided by an ideal lossless inductor--capacitor network. Associate to each vertex $i$ a capacitor of capacitance $\kappa_i>0$ connected to ground, and to each edge $\{i,j\}$ an inductor of inductance $\ell_{ij}>0$. If $V_i(t)$ denotes the voltage at vertex $i$, Kirchhoff's laws lead to
$$
D_\kappa V_{tt}+L_\omega V=0,
\qquad
D_\kappa=\operatorname{diag}(\kappa_1,\ldots,\kappa_N),
\qquad
\omega_{ij}=\ell_{ij}^{-1},
$$
where $L_\omega$ is the positively weighted graph Laplacian. The corresponding graph wave model for an inductor--capacitor network is derived in Section~2 of \cite{CaputoKnippelSimo}. The graph notation used here is introduced in Section~\ref{sec:background}, and the positively weighted Laplacian is defined in Section~\ref{sec:weighted}. Setting $c_i=\kappa_i^{-1/2}$ gives the weighted graph wave equation considered below. In this interpretation, recovering the initial signal amounts to reconstructing an unknown initial voltage distribution from measurements at selected circuit nodes, while recovering the vertex wave speeds is equivalent to determining the unknown nodal capacitances. The choice of observation vertices can therefore also be viewed as a sensor placement problem.

Inverse problems on finite graphs have also been studied under substantially different data models. Draper et al.~\cite{Draper} consider the recovery of matrix-valued edge or nodal parameters from boundary-response data, with applications to mass-spring-damper networks, and analyze the Jacobian of the parameter-to-data map and questions of local identifiability. Bl{\aa}sten et al.~\cite{BlastenIsozakiLassasLu} study a discrete Gelfand inverse problem in which the graph structure and weights are recovered from boundary spectral data under geometric hypotheses. More recently, Li et al.~\cite{LiGaoGengYang} considered recovery of interior vertex weights on finite weighted graphs from Neumann boundary spectral data using a discrete boundary control approach. Our data model is different. The graph geometry is known, the unknown coefficients are scalar vertex wave speeds, and the available data consist only of the time-dependent trace generated by a single prescribed initial displacement on a set of observation vertices. We show that this single experiment determines all vertex wave speeds under a zero forcing condition and the nonharmonicity assumption
$$
(Lf)_i\neq0
\qquad
\text{for every }i\in V.
$$
The same propagation mechanism also yields injectivity of the linearized measurement map and Lipschitz stability on compact admissible classes.

The inverse problems considered here exhibit discrete counterparts of familiar mechanisms from the continuum. Observability governs recovery of the initial state, while propagation from the observation set together with a nondegeneracy condition on the initial signal governs recovery of the vertex wave speeds. We establish uniqueness and stability results for both problems. For recovery of the initial signal, we relate uniqueness to observability of the associated finite-dimensional wave system and identify conditions on the observation set that guarantee recovery. For recovery of the vertex wave speeds, we show that a single experiment is sufficient under the zero forcing condition and the nonharmonicity assumption above. In particular, the uniqueness result holds on every observation interval of positive length, although the stability constants and numerical performance may depend on the observation time.

We also develop a nonlinear least-squares method for reconstructing the vertex wave speeds and analyze its consistency under temporal discretization. Numerical experiments illustrate the reconstruction and examine the effects of observation geometry, time horizon, coefficient localization, and measurement noise on its performance. The main results and the reconstruction framework extend to positively weighted graphs.

The semidiscretization above also suggests a conditional link back to the continuous inverse problems. Suppose that a family of spatial discretizations is convergent for the forward problem. If the associated discrete inverse stability constants remain uniformly bounded as $h\to0$, then forward consistency together with the stability estimates proved here yields convergence of the discrete reconstructions to the corresponding continuous unknown. Such uniformity is not automatic. For standard finite-difference semidiscretizations of the wave equation, high-frequency numerical modes can cause the observability constants to deteriorate as the mesh is refined \cite{MaricaZuazua}. Establishing uniform inverse stability for a particular family of discretizations is a separate problem and is not pursued here.

The remainder of the paper is organized as follows. Section~\ref{sec:background} introduces the graph notation and the continuous and discrete wave models. Section~\ref{sec:source-recovery} treats initial-signal recovery and its stability. Section~\ref{sec:speed-recovery} establishes uniqueness and stability for speed recovery and presents the reconstruction algorithm. Section~\ref{sec:weighted} extends the results to edge-weighted graphs. Section~\ref{sec:numerics} presents the numerical implementation and experiments.

\section{Preliminaries}
\label{sec:background}

In this section, we introduce the graph-theoretic notation and the continuous and discrete wave models used throughout the paper. We also define the observation sets and graph-theoretic notions that will play a central role in the inverse problems studied in the subsequent sections.

\subsection{Graph notation}

We begin with the basic notation for finite graphs. Throughout the paper, all graphs are assumed to be finite, simple, and undirected.

\begin{defn}
An ordered pair $\Gamma=(V,E)$ is a simple undirected graph if
$$
V=\{1,\ldots,N\}\subseteq\mathbb{N}
$$
and
$$
E\subseteq\bigl\{\{i,j\}\mid i,j\in V,\ i<j\bigr\}.
$$
The set $V$ is called the vertex set and $E$ the edge set.
\end{defn}

\begin{defn}
The adjacency matrix $A_\Gamma\in\mathbb{R}^{N\times N}$ is defined by
$$
(A_\Gamma)_{ij}:=
\begin{cases}
1, & \text{if }\{i,j\}\in E,\\
0, & \text{otherwise}.
\end{cases}
$$
The degree matrix is
$$
D_\Gamma=\operatorname{diag}(d_1,\ldots,d_N),
\qquad
d_i=\sum_{j=1}^N(A_\Gamma)_{ij}.
$$
\end{defn}

\begin{defn}
The graph Laplacian is the matrix
$$
L=D_\Gamma-A_\Gamma.
$$
\end{defn}

For a simple graph, the off-diagonal entries of $L$ determine adjacency and the diagonal entries determine the vertex degrees. Thus $L$ determines the graph and will be regarded throughout as encoding its geometry.

A scalar graph signal on $\Gamma$ is a vector $X\in\mathbb{R}^N$, where $X_i$ denotes the value associated with vertex $i$. Throughout this paper, we consider scalar graph signals.

\begin{defn}
Let $B\subseteq V$. Initially color the vertices in $B$ black and the remaining vertices white. If a black vertex has exactly one white neighbor, that neighbor may be changed to black. The set $B$ is called a zero forcing set if repeated application of this rule eventually colors every vertex black. A zero forcing sequence is then a sequence $(S_k)_{k=0}^K$ with $S_0=B$ and $S_K=V$ such that, at each stage, a vertex $u\in S_k$ forces a vertex $w\in S_{k+1}\setminus S_k$. In this case one may say that $S_k$ forces $S_{k + 1}$ and $u$ forces $w$.  
\end{defn}

The zero forcing condition will be used below to describe how information measured on the observation set can propagate through the graph. There is extensive literature on the zero forcing condition and its applications to the minimum rank problem and the inverse eigenvalue problem for graphs. For a comprehensive treatment of these subjects, see \cite{HogbenLinShader2022}. 

\begin{defn}
Let $\Gamma$ be a graph with graph Laplacian $L$, and let
$$
X:[0,\infty)\to\mathbb{R}^N
$$
be a time-dependent graph signal. The graph wave equation is
\begin{equation}
\left\{
\begin{aligned}
X_{tt}(t) &= -C^2LX(t),
    && t\geq0,\\
X(0) &= f,\\
X_t(0) &= 0,
\end{aligned}
\right.
\label{eq:graph-wave-equation}
\end{equation}
where
$$
C=\operatorname{diag}(c_1,\ldots,c_N)
$$
is the positive diagonal matrix of vertex wave speeds and $f\in\mathbb{R}^N$ is the initial signal.
\end{defn}

Equation \eqref{eq:graph-wave-equation} is a finite-dimensional system of ordinary differential equations in time, with the graph Laplacian describing the spatial coupling between vertices. The associated inverse problems nevertheless retain nontrivial questions of identifiability and stability.

The sign in \eqref{eq:graph-wave-equation} reflects the convention used for the graph Laplacian. The matrix $L=D_\Gamma-A_\Gamma$ is positive semidefinite, whereas the continuous Laplacian
$$
\Delta=\sum_{k=1}^d\frac{\partial^2}{\partial x_k^2}
$$
is negative semidefinite on the usual function spaces. Thus $-L$ plays the role of the continuous operator $\Delta$.

Since \eqref{eq:graph-wave-equation} is a finite-dimensional linear system with constant coefficients, standard ODE theory guarantees a unique solution for every initial signal $f\in\mathbb{R}^N$ and every positive diagonal speed matrix $C$. See, for example, \cite{ODE}.

\begin{defn}
Let $\Gamma=(V,E)$ be a graph with known geometry, and let $L$ denote its graph Laplacian. Let $B\subseteq V$ be a designated set of observation vertices, and let
$$
R_B:\mathbb{R}^N\to\mathbb{R}^{|B|}
$$
denote restriction to the coordinates indexed by $B$. For any admissible initial signal $f$ and positive diagonal speed matrix $C$, let $X$ be the solution of \eqref{eq:graph-wave-equation}. The graph observation map is defined by
$$
\Phi_B(f,C,L):=R_BX.
$$
Throughout the inverse problems considered in this paper, the graph geometry is assumed known, and hence $L$ is fixed. When the vertex wave speeds are also known, we write
$$
\Phi_B^{C,L}(f):=\Phi_B(f,C,L)
$$
for the map from the initial signal to the observed time trace. When the initial signal is known, we write
$$
\Phi_B^{f,L}(C):=\Phi_B(f,C,L)
$$
for the map from the vertex wave speeds to the observed time trace.
\end{defn}
	
\section{Initial-signal recovery}

\label{sec:source-recovery}

We first consider recovery of the initial signal $f$ when the speed matrix $C$ and graph Laplacian $L$ are known. This problem can be formulated within the classical theory of observability for finite-dimensional linear systems. We derive the corresponding observability criterion, an explicit reconstruction formula, and a stability estimate for the graph wave equation. We then relate observability to the geometry of the observation set and show that every zero forcing observation set guarantees observability for any positive diagonal speed matrix $C$. Zero forcing has previously been used in the study of controllability of graph-structured dynamical systems \cite{BurgarthEtAl,MonshizadehEtAl}, while here it is used in the inverse problem of recovering the initial signal from partial time-dependent measurements.

\subsection{Observability and reconstruction}\label{section:f-unique}

	Let $\Gamma$ be an undirected, connected graph with graph Laplacian $L$, and let
$$
C=\operatorname{diag}(c_1,\ldots,c_N),
\qquad c_i>0.
$$
Suppose that $X(t)$ satisfies the graph wave equation \eqref{eq:graph-wave-equation}, and define the state vector
$$
x(t):=
\begin{pmatrix}
X(t)\\
X_t(t)
\end{pmatrix},
\qquad
A:=
\begin{pmatrix}
0 & I_N\\
-C^2L & 0
\end{pmatrix}.
$$
Let $B\subseteq V$ be the set of observation vertices, and let
$$
R_B:\mathbb{R}^N\to\mathbb{R}^{|B|}
$$
denote restriction to the coordinates indexed by $B$. The measured output is
$$
y(t):=R_BX(t).
$$
If
$$
H_B:=R_B
\begin{pmatrix}
I_N & 0
\end{pmatrix},
$$
then $A\in\mathbb{R}^{2N\times2N}$, $H_B\in\mathbb{R}^{|B|\times2N}$, and
\eqref{eq:graph-wave-equation} can be written as the time-invariant finite-dimensional linear output system
\begin{equation}
\left\{
\begin{aligned}
\dot{x}(t) &= Ax(t),\\
y(t) &= H_Bx(t).
\end{aligned}
\right.
\label{eq:linear-output-system}
\end{equation}
    We recall the standard observability notions needed below.
	\begin{defn}
		A pair $(A,H_B)$ is observable if there exists a finite $t_1>0$ such that knowledge of $y$ on $[0,t_1)$ uniquely determines every initial state $x(0)$. Equivalently, $H_Be^{At}v=0$ for all $t\in[0,t_1)$ implies $v=0$. 
	\end{defn}
    
	\begin{defn}
		The observability Gramian of a linear system is the matrix
		\begin{equation*}
			W_o(t_1) := \int_0^{t_1} e^{A^{T} t} H_B^TH_Be^{At} \; d t
		\end{equation*}
	\end{defn}

    The observability Gramian is symmetric and positive semidefinite. In the present setting it is a $2N \times 2N$ matrix. Its quadratic form measures the total squared magnitude of the measured signal over the interval $[0, t_1)$. The Gramian is invertible precisely when every nonzero initial state produces a nonzero measured signal. Thus, it fails to be invertible exactly when there is a nontrivial initial state whose measured trace vanishes identically on $[0,t_1)$. Small eigenvalues of the Gramian correspond to initial states that are only weakly visible in the observations.

    We use the following standard characterization of observability.
    
	\begin{thm}[\cite{Chen} Theorem 6.O1, p. 156]
        \label{thm:observability-characterization}
		The following are equivalent:
		\begin{enumerate}
			\item the pair $(A, H_B)$ is observable; 
			\item for some $t_1 > 0$, the observability Gramian is nonsingular; 
			\item the matrix 
			\begin{equation*}
				\begin{pmatrix}
					A - \lambda I \\
					H_B
				\end{pmatrix}
			\end{equation*}
			has full column rank for every eigenvalue $\lambda$ of $A$.
		\end{enumerate}
\end{thm}

Previous work has shown that zero forcing sets provide sufficient sets of control vertices for broad classes of linear dynamical systems whose coupling structure is described by a graph \cite{BurgarthEtAl,MonshizadehEtAl}. The result below gives the corresponding observation statement for the graph wave equation considered here. Namely, we show that measurements on any zero forcing set determine the initial state for every positive diagonal speed matrix $C$. We include a direct proof in the present setting.
    \begin{prop}[Zero forcing implies observability]
        \label{prop:zero-forcing-observability}
        Let $\Gamma=(V,E)$ be an undirected graph, let $C=\diag(c_1,\dots,c_N)$ with $c_i>0$, and let $B\subseteq V$ be a zero forcing set. Then the pair $(A,H_B)$ associated with the graph wave equation is observable.
    \end{prop}
   \begin{proof}
We verify the rank condition in Theorem~\ref{thm:observability-characterization}. Let $\lambda\in\mathbb{C}$ be an eigenvalue of $A$ and suppose that $p,q\in\mathbb{C}^N$ satisfy
$$
(A-\lambda I)\binom{p}{q}=0,
\qquad
H_B\binom{p}{q}=0.
$$
From the definitions of $A$ and $H_B$,
$$
q=\lambda p,
\qquad
-C^2Lp=\lambda q=\lambda^2p,
\qquad
R_Bp=0.
$$
Hence $p_i=0$ for every $i\in B$ and
\begin{equation}
C^2Lp=-\lambda^2p.
\label{eq:zf-observability-eigenvector}
\end{equation}

Choose a zero forcing process starting from $B$ and, if several forces occur at the same stage, order them arbitrarily. Let $S_k$ be the colored set after $k$ individual forces, so that $S_0=B$, $S_K=V$, and
$$
S_{k+1}=S_k\cup\{w\}
$$
whenever $u\in S_k$ forces $w$. We prove inductively that $p_i=0$ for every $i\in S_k$. The assertion holds for $S_0$ because $R_Bp=0$.

Suppose it holds on $S_k$, and let $u\in S_k$ force $w$. Since $p_u=0$, evaluating \eqref{eq:zf-observability-eigenvector} at $u$ gives
$$
c_u^2(Lp)_u=-\lambda^2p_u=0.
$$
Because $c_u>0$, it follows that $(Lp)_u=0$. Moreover, $w$ is the unique neighbor of $u$ outside $S_k$, while every vertex in $S_k$ has zero $p$-component by the induction hypothesis. Therefore
$$
0=(Lp)_u=\deg(u)p_u-\sum_{j\sim u}p_j=-p_w,
$$
and hence $p_w=0$. Thus the zero set propagates from $S_k$ to $S_{k+1}$. By induction, $p=0$ on all of $V$. Since $q=\lambda p$, we also have $q=0$.

Therefore
$$
\ker
\begin{pmatrix}
A-\lambda I\\
H_B
\end{pmatrix}
=\{0\}
$$
for every eigenvalue $\lambda$ of $A$. By the rank condition in Theorem~\ref{thm:observability-characterization}, the pair $(A,H_B)$ is observable.
\end{proof}

    \begin{rmk}
Zero forcing is a sufficient graph-theoretic condition for observability in Proposition~\ref{prop:zero-forcing-observability}, but no converse is asserted. Related uses of zero forcing in the controllability of graph-structured linear systems appear in \cite{BurgarthEtAl,MonshizadehEtAl}. Here, zero forcing provides a condition on the observation set that guarantees recovery of the initial signal. The same graph-theoretic condition will also play a central role in the coefficient-recovery problem, together with an additional nonharmonicity condition on the initial signal.
    \end{rmk}
    
\begin{lemma}
\label{lem:displacement-observability}
For the graph wave system with $X_t(0)=0$, injectivity of the map $f\mapsto R_BX$ on an interval of positive length is equivalent to observability of the associated first-order pair $(A,H_B)$.
\end{lemma}

\begin{proof}
If $(A,H_B)$ is observable, then by Theorem~\ref{thm:observability-characterization} the initial state is uniquely determined by the measured output. In particular, the map $f\mapsto R_BX$ is injective on the restricted class of initial states $(f,0)^T$.

Conversely, assume that $f\mapsto R_BX$ is injective, and let $x_0=(f,g)^T$ be an unobservable initial state for the first-order system. Since $x_0$ is unobservable,
$$
H_Be^{At}x_0=0
$$
on the observation interval. Differentiating at $t=0$ gives
$$
H_BA^kx_0=0
\qquad
\text{for every }k\geq0.
$$
Set $K=C^2L$. Since
$$
A^{2r}=(-1)^r
\begin{pmatrix}
K^r&0\\
0&K^r
\end{pmatrix},
\qquad
A^{2r+1}=(-1)^r
\begin{pmatrix}
0&K^r\\
-K^{r+1}&0
\end{pmatrix},
$$
it follows that
$$
R_BK^rf=0,
\qquad
R_BK^rg=0,
\qquad
r\geq0.
$$
Using the power-series expansion of $e^{At}$, the first family of identities shows that the solution with initial state $(f,0)^T$ has identically vanishing measured trace. The second family gives the same conclusion for the solution with initial state $(g,0)^T$. Injectivity of $f\mapsto R_BX$ therefore gives $f=g=0$. Hence the only unobservable initial state is the zero state, and $(A,H_B)$ is observable.
\end{proof}
    
	\begin{prop}
		If $(A, H_B)$ is observable, the measurement map $\Phi_B^{C,L}:f\mapsto R_BX$ is injective on any interval of positive length. Moreover, 
		\begin{equation}
			\label{eq:initialsignal}
			x(0) = W_o^{-1}(t_1) \int_0^{t_1} e^{A^T t} H_B^T y(t) \; d t.
		\end{equation}
	\end{prop}
	\begin{proof}
		Since $t\mapsto H_Be^{At}v$ is real analytic, vanishing on any interval of positive length implies vanishing identically. Thus the injectivity statement follows from Lemma \ref{lem:displacement-observability}. Noting that 
		\begin{equation*}
			y(t) = H_Be^{At}x(0), 
		\end{equation*}
		multiplying by $e^{ A^Tt}H_B^T$ and integrating gives
		\begin{equation*}
			\int_0^{t_1} e^{A^T t} H_B^T y(t) \; dt =\left(\int_0^{t_1} e^{A^Tt} H_B^T H_B e^{At} \; dt\right)x(0) = W_o(t_1)x(0). 
		\end{equation*}
		Hence
        \begin{equation*}
			x(0)  = W_o^{-1}(t_1) \int_0^{t_1} e^{A^T t} H_B^T y(t) \; d t.
		\end{equation*}
	\end{proof}
	
    \begin{rmk} A naive dense implementation that approximates the observability Gramian using $K$ time points, evaluates a $2N\times2N$ matrix exponential at each time point, and then forms and solves with the dense Gramian has the crude upper bound $O(K(2N)^3)$. This estimate is implementation-dependent, and exploiting matrix structure or reusing factorizations can reduce the computational cost.
    \end{rmk}

\subsection{Stability under measurement noise}
    Because the reconstruction operator is explicit, its sensitivity to measurement noise can be estimated directly.
    
	Let $n \in L^2(0, t_1; \mathbb{R}^{|B|})$ be a noise signal and suppose 
	\begin{equation*}
		y(t) = H_Be^{At}x(0) + n(t).
	\end{equation*} 
	Applying \eqref{eq:initialsignal} to the perturbed data gives
	\begin{equation*}
		x_e(0) = W_o^{-1}(t_1)\int_0^{t_1} e^{A^T t} H_B^T y(t) \; dt 
	\end{equation*}
	whereas the exact data give
	\begin{equation*}
		x(0) = W_o^{-1}(t_1)\int_0^{t_1} e^{A^T t} H_B^T (y(t) - n(t)) \; dt
	\end{equation*}
	Setting $z := x_e(0) - x(0)$, the goal is to exhibit a bound on $\norm{z}_{\ell^2}$. Begin by noting that
	\begin{equation*}
		W_o(t_1)z = \int_0^{t_1} e^{A^T t} H_B^T n(t) \; dt.
	\end{equation*}
	Taking $\langle z, - \rangle$ of both sides, one has on the one hand that
	\begin{align*}
		\langle z, W_o(t_1) z \rangle & = \left \langle z, \int_0^{t_1} e^{A^Tt}H_B^T n(t) \; dt \right \rangle \\
		& = \int_0^{t_1} \langle z, e^{A^T t}H_B^T n(t) \rangle \; dt \\
		& =  \int_0^{t_1} \langle H_B e^{A t}z,  n(t) \rangle \; dt. 
	\end{align*} 
	By the definition of the Gramian,
	\begin{align*}
		\langle z, W_o(t_1)z \rangle & = z^T W_o(t_1)z \\
		& =  z^T \left( \int_0^{t_1} e^{A^{T} t} H_B^TH_Be^{At} \; dt \right) z \\
		& = \int_0^{t_1} z^T e^{A^{T} t} H_B^TH_Be^{At} z \; dt \\
		& = \int_0^{t_1} \norm{H_B e^{At}z}_{\ell^2}^2 \; dt.
	\end{align*}
	Thus 
	\begin{equation}
		\label{eq:energy-identity-stability}
		\int_0^{t_1} \norm{H_B e^{At}z}_{\ell^2}^2 \; dt = \int_0^{t_1} \langle H_B e^{A t}z,  n(t) \rangle \; dt.
	\end{equation}
	The right-hand side is bounded by 
	\begin{align*}
		\left| \int_0^{t_1} \langle H_B e^{A t}z,  n(t) \rangle \; dt \right| &  \leq \left(\int_0^{t_1} \norm{H_Be^{A t} z}_{\ell^2}^2 \; dt \right)^{1/2}\left( \int_0^{t_1} \norm{n}_{\ell^2}^2 \; dt \right)^{1/2} \\
		&  =  \left(\int_0^{t_1} \norm{H_Be^{A t} z}_{\ell^2}^2 \; dt \right)^{1/2} \norm{n}_{L^2(0, t_1; \mathbb{R}^{|B|})} .
	\end{align*}
	Hence from \eqref{eq:energy-identity-stability}
	\begin{equation*}
		\left(\int_0^{t_1} \norm{H_Be^{A t} z}_{\ell^2}^2 \; dt \right)^{1/2} \leq \norm{n}_{L^2([0, t_1]; \mathbb{R}^{|B|})}.
	\end{equation*}
	Since observability of $(A,H_B)$ implies that $W_o(t_1)$ is positive definite,
	\begin{equation*}
		z^T W_o(t_1)z \geq \lambda_{\min}(W_o(t_1)) \norm{z}_{\ell^2}^2.
	\end{equation*}
	Thus 
	\begin{equation*}
		\lambda_{\min}(W_o(t_1)) \norm{z}_{\ell^2}^2 \leq  \int_0^{t_1} \norm{H_Be^{A t} z}_{\ell^2}^2 \; dt \leq \norm{n}_{L^2(0, t_1; \mathbb{R}^{|B|})}^2
	\end{equation*}
	and therefore
	\begin{equation}
		\norm{x_e(0) - x(0)}_{\ell^2}  \leq \frac{\norm{n}_{L^2(0, t_1; \mathbb{R}^{|B|})}}{\sqrt{\lambda_{\min}(W_o(t_1))}}.
	\end{equation}
	Thus the smallest eigenvalue of the observability Gramian controls the stability of the reconstruction. Small values correspond to states that are weakly visible in the measurements. The reconstruction of such states requires strongly amplifying the measured data and therefore the noise signal. As $\lambda_{\min}(W_o(t_1))$ approaches zero, the estimate degenerates consistently with loss of observability.
\section{Vertex-speed recovery}

\label{sec:speed-recovery}

We next recover the vertex wave speeds $C$ from a known initial signal $f$, known graph geometry $L$, and time-dependent measurements on the observation set. The uniqueness argument is graph-theoretic and uses zero forcing to propagate equality of the coefficients from the measured vertices to the rest of the network. The same mechanism also proves injectivity of the Fr\'echet derivative of the measurement map. Combined with a finite-dimensional stability result, this gives a Lipschitz stability estimate on compact admissible classes. We then formulate a nonlinear least-squares reconstruction method and discuss its conditioning.

\subsection{Uniqueness}

Let $\Gamma$ be an undirected, connected graph with graph Laplacian $L$, and let $f\in\mathbb{R}^N$ be a known initial signal. Fix $T>0$, and suppose
\begin{equation*}
\begin{cases}
X_{tt}(t)=-C^2LX(t) & \text{on } [0,T],\\
X(0)=f, \quad X_t(0)=0.
\end{cases}
\end{equation*}
For an observation set $B\subseteq V$, the corresponding measurement map is
\begin{equation*}
\Phi_{B,T}^{f,L}:(0,\infty)^N
\longrightarrow
L^2(0,T;\mathbb{R}^{|B|}),
\end{equation*}
which maps the vertex speeds to the observed trace $R_BX$.

For the analysis, it is convenient to reparameterize the vertex speeds by
$$
M=\diag(m_1,\ldots,m_N),
\qquad
m_i=c_i^{-2}.
$$
The graph wave equation then takes the equivalent form
\begin{equation}
\label{eq:mass-form}
\begin{cases}
MX_{tt}(t)=-LX(t) & \text{on } [0,T],\\
X(0)=f, \quad X_t(0)=0.
\end{cases}
\end{equation}
This is the discrete counterpart of the continuous mass formulation
$$
m(x)u_{tt}-\Delta u=0,
\qquad
m(x)=c(x)^{-2},
$$
introduced in the Introduction. We refer to the coefficients $m_i=c_i^{-2}$ as the \emph{mass-form variables}.

Recall that $B$ is a zero forcing set if repeated application of the zero forcing rule starting from $B$ eventually reaches every vertex of $\Gamma$. When needed, we write
$$
S_0=B,\qquad S_K=V
$$
for the successive colored sets in a zero forcing sequence, where at each step a vertex $u\in S_k$ forces its unique neighbor $w\notin S_k$, so that
$$
S_{k+1}=S_k\cup\{w\}.
$$

\begin{rmk}
The role of zero forcing here is reminiscent of the convex-foliation condition used for wave-speed recovery in the continuous setting \cite{StefanovUhlmann}. In both cases, uniqueness is propagated from the observed region into the unknown region one layer at a time. A zero forcing sequence provides a discrete mechanism for this propagation through the graph.
\end{rmk}

\begin{rmk}
The injectivity result below requires the nonharmonicity condition
$$
(Lf)_i\neq0
\qquad
\text{for every }i\in V.
$$
This condition is nonrestrictive. For each fixed vertex $i$, the condition $(Lf)_i=0$ defines a proper hyperplane in $\mathbb{R}^N$. Hence the set of initial signals for which $(Lf)_i=0$ at some vertex is a finite union of hyperplanes and therefore has Lebesgue measure zero. Thus the condition holds for almost every $f\in\mathbb{R}^N$. Moreover, $(Lf)_i=0$ is equivalent to
$$
f_i=\frac{1}{\deg(i)}\sum_{j\sim i}f_j,
$$
so $f_i$ is the average of its neighboring values. Thus $(Lf)_i\neq0$ is the discrete analogue of a local nonharmonicity condition. See Theorem 3.3 and Remark 3.2 of \cite{StefanovUhlmann} for an analogous condition in the continuous wave setting.
\end{rmk}

\begin{thm}[Uniqueness]
\label{thm:one-experiment-non-harmonic}

Suppose that $B\subseteq V$ is a zero forcing set of $\Gamma$, and let $T>0$ and $f\in\mathbb{R}^N$ satisfy
$$
(Lf)_i\neq0
\qquad
\text{for every }i\in V.
$$
Let
$$
M=\diag(m_1,\ldots,m_N),
\qquad
\widetilde{M}=\diag(\widetilde{m}_1,\ldots,\widetilde{m}_N),
$$
with $m_i,\widetilde{m}_i>0$ for every $i\in V$, and let $X$ and $\widetilde{X}$ denote the corresponding solutions of \eqref{eq:mass-form} with the same initial signal $f$. If
$$
R_BX(t)=R_B\widetilde{X}(t)
\qquad
\text{for every }t\in[0,T],
$$
then
$$
M=\widetilde{M}.
$$
\end{thm}

\begin{proof}
We proceed along a zero forcing sequence.

\emph{Initialization.}
For each $i\in B$, the equality of the observed traces gives
$$
X_i(t)=\widetilde{X}_i(t)
\qquad
\text{for all }t\in[0,T],
$$
and hence
$$
X_{tt,i}(0)=\widetilde{X}_{tt,i}(0).
$$
Evaluating the two equations at $t=0$ and using $X(0)=\widetilde{X}(0)=f$ gives
$$
m_iX_{tt,i}(0)=-(Lf)_i,
\qquad
\widetilde{m}_i\widetilde{X}_{tt,i}(0)=-(Lf)_i.
$$
Since $(Lf)_i\neq0$, the common acceleration is nonzero, and therefore
$$
m_i=\widetilde{m}_i
\qquad
\text{for every }i\in B.
$$

\emph{Propagation.}
Let
$$
S_0=B,\qquad S_K=V
$$
be the successive colored sets in a zero forcing sequence. Set
$$
W=\widetilde{X}-X.
$$
Subtracting the two equations gives
\begin{equation}
MW_{tt}+LW=(M-\widetilde{M})\widetilde{X}_{tt}.
\label{eq:forced-wave}
\end{equation}

Suppose inductively that
$$
W_j\equiv0
\qquad\text{and}\qquad
m_j=\widetilde{m}_j
$$
for every $j\in S_k$. Let $u\in S_k$ force $w\notin S_k$. Evaluating \eqref{eq:forced-wave} at $u$ and using the induction hypothesis yields
$$
(LW)_u=0
$$
on $[0,T]$. Since $w$ is the unique neighbor of $u$ outside $S_k$, all other neighbors of $u$ belong to $S_k$. Therefore
$$
(LW)_u
=
\deg(u)W_u-\sum_{j\sim u}W_j
=
-W_w.
$$
It follows that
$$
W_w\equiv0
$$
on $[0,T]$, so
$$
X_w(t)=\widetilde{X}_w(t)
\qquad
\text{for all }t\in[0,T].
$$
In particular,
$$
X_{tt,w}(0)=\widetilde{X}_{tt,w}(0).
$$
Evaluating the two equations at $w$ and $t=0$ gives
$$
m_wX_{tt,w}(0)=-(Lf)_w,
\qquad
\widetilde{m}_w\widetilde{X}_{tt,w}(0)=-(Lf)_w.
$$
Since $(Lf)_w\neq0$, the common acceleration is nonzero, and hence
$$
m_w=\widetilde{m}_w.
$$
Thus both equality of the traces and equality of the corresponding coefficients propagate from $S_k$ to $S_{k+1}$. Iterating along the zero forcing sequence gives
$$
M=\widetilde{M}.
$$
\end{proof}

\subsection{Lipschitz stability}

We next establish stability of the vertex-speed recovery with respect to perturbations of the measured trace. We identify a positive diagonal matrix
$$
M=\diag(m_1,\ldots,m_N)
$$
with its diagonal vector in $\mathbb{R}^N$, and set
$$
\mathcal U=(0,\infty)^N,
\qquad
\mathcal K=[m_{\min},m_{\max}]^N\subset\mathcal U,
$$
where
$$
0<m_{\min}<m_{\max}<\infty.
$$
For each $M\in\mathcal U$, let $X_M$ denote the solution of \eqref{eq:mass-form}. With $f$, $L$, $B$, and $T>0$ fixed, define the parameter-to-data map
$$
\Phi_T(M)=R_BX_M
\in L^2(0,T;\mathbb{R}^{|B|}).
$$

We first compute the Fr\'echet derivative of $\Phi_T$. The first-order formulation of \eqref{eq:mass-form} has coefficient matrix
$$
A(M)=
\begin{pmatrix}
0&I_N\\
-M^{-1}L&0
\end{pmatrix},
$$
which depends smoothly on $M\in\mathcal U$. Since solutions of finite-dimensional linear systems depend smoothly on their coefficients, the map $M\mapsto X_M$ is smooth on $\mathcal U$, and hence so is $\Phi_T$.

Fix a diagonal perturbation
$$
H=\diag(h_1,\ldots,h_N),
$$
and define
$$
Z(t)
=
\left.\frac{d}{d\varepsilon}X_{M+\varepsilon H}(t)\right|_{\varepsilon=0}.
$$
Differentiating
$$
(M+\varepsilon H)(X_{M+\varepsilon H})_{tt}
+
LX_{M+\varepsilon H}=0
$$
at $\varepsilon=0$ gives
$$
MZ_{tt}+LZ=-H(X_M)_{tt}.
$$
Since the initial data are independent of $M$, we also have
$$
Z(0)=0,
\qquad
Z_t(0)=0.
$$
Therefore
\begin{equation}
\label{eq:frechet-derivative}
D\Phi_T(M)[H]=R_BZ,
\end{equation}
where $Z$ solves
\begin{equation}
\label{eq:variation-wave-equation}
\begin{cases}
MZ_{tt}+LZ=-H(X_M)_{tt} & \text{on }[0,T],\\
Z(0)=0,\quad Z_t(0)=0.
\end{cases}
\end{equation}

\begin{lemma}[Injectivity of the Fr\'echet derivative]
\label{lemma:frechet-injectivity}
Suppose that $B\subseteq V$ is a zero forcing set of $\Gamma$ and
$$
(Lf)_i\neq0
\qquad
\text{for every }i\in V.
$$
Then $D\Phi_T(M)$ is injective for every $M\in\mathcal U$.
\end{lemma}

\begin{proof}
Let $H=\diag(h_1,\ldots,h_N)$ and suppose
$$
D\Phi_T(M)[H]=0.
$$
By \eqref{eq:frechet-derivative},
$$
R_BZ=0
$$
in $L^2(0,T;\mathbb{R}^{|B|})$. Since $Z$ is smooth in time, this implies
$$
Z_i\equiv0
\qquad
\text{for every }i\in B.
$$
Evaluating \eqref{eq:variation-wave-equation} at $i\in B$ and $t=0$ gives
$$
0=-h_i(X_M)_{tt,i}(0).
$$
Since
$$
(X_M)_{tt,i}(0)=-m_i^{-1}(Lf)_i\neq0,
$$
we obtain
$$
h_i=0
\qquad
\text{for every }i\in B.
$$

We now propagate these two identities along a zero forcing sequence. Suppose that
$$
Z_j\equiv0
\qquad\text{and}\qquad
h_j=0
$$
for every $j\in S_k$, and let $u\in S_k$ force $w\notin S_k$. Evaluating \eqref{eq:variation-wave-equation} at $u$ gives
$$
(LZ)_u=0.
$$
Since $w$ is the unique neighbor of $u$ outside $S_k$,
$$
(LZ)_u
=
\deg(u)Z_u-\sum_{j\sim u}Z_j
=
-Z_w.
$$
Hence
$$
Z_w\equiv0.
$$
Evaluating \eqref{eq:variation-wave-equation} at $w$ and $t=0$ then gives
$$
0=-h_w(X_M)_{tt,w}(0).
$$
Again,
$$
(X_M)_{tt,w}(0)=-m_w^{-1}(Lf)_w\neq0,
$$
so $h_w=0$. Iterating along the zero forcing sequence gives $H=0$.
\end{proof}

We use the following stability result of Bourgeois.

\begin{prop}[\cite{Bourgeois}, Theorem 2.1]
Let $\mathcal U$ be an open subset of a finite-dimensional Hilbert space, let $\mathcal K\subset\mathcal U$ be compact and convex, and let
$$
G\in C^1(\mathcal U;\mathcal H)
$$
for a Hilbert space $\mathcal H$. If $G$ is injective on $\mathcal U$ and $DG(M)$ is injective for every $M\in\mathcal U$, then there exists $K_{\mathrm{stab}}>0$ such that
$$
\norm{M-\widetilde M}
\leq
K_{\mathrm{stab}}
\norm{G(M)-G(\widetilde M)}
$$
for every $M,\widetilde M\in\mathcal K$.
\end{prop}

Under the zero forcing and nonharmonicity assumptions, Theorem \ref{thm:one-experiment-non-harmonic} gives injectivity of $\Phi_T$ on $\mathcal U$, while Lemma \ref{lemma:frechet-injectivity} gives injectivity of its derivative. Therefore there exists $K_{\mathrm{stab}}>0$ such that
\begin{equation}
\label{eq:mass-recovery-stability}
\norm{M-\widetilde M}_{\ell^2}
\leq
K_{\mathrm{stab}}
\norm{\Phi_T(M)-\Phi_T(\widetilde M)}_{L^2(0,T;\mathbb{R}^{|B|})}
\end{equation}
for every $M,\widetilde M\in\mathcal K$.

Now suppose that the measured data are
$$
y_e=\Phi_T(M)+n,
$$
where
$$
n\in L^2(0,T;\mathbb{R}^{|B|})
$$
is the measurement noise. Let $M_e\in\mathcal K$ be a reconstruction and define its residual by
$$
r=\Phi_T(M_e)-y_e.
$$
Then
$$
\Phi_T(M_e)-\Phi_T(M)=n+r.
$$
Applying \eqref{eq:mass-recovery-stability} and the triangle inequality gives
\begin{equation}
\label{eq:mass-recovery-stability-noise}
\norm{M-M_e}_{\ell^2}
\leq
K_{\mathrm{stab}}
\left(
\norm{n}_{L^2(0,T;\mathbb{R}^{|B|})}
+
\norm{r}_{L^2(0,T;\mathbb{R}^{|B|})}
\right).
\end{equation}
The estimate can be expressed in terms of the original vertex speeds. Since
$$
c_i=m_i^{-1/2}
$$
and $m_i,m_{e,i}\in[m_{\min},m_{\max}]$, the mean value theorem gives
$$
|c_i-c_{e,i}|
\leq
\frac{1}{2m_{\min}^{3/2}}
|m_i-m_{e,i}|.
$$
Consequently,
\begin{equation}
\label{eq:speed-recovery-stability}
\norm{C-C_e}_{\ell^2}
\leq
\frac{K_{\mathrm{stab}}}{2m_{\min}^{3/2}}
\left(
\norm{n}_{L^2(0,T;\mathbb{R}^{|B|})}
+
\norm{r}_{L^2(0,T;\mathbb{R}^{|B|})}
\right).
\end{equation}
In particular, if the reconstruction fits the measured data exactly, then $r=0$ and
$$
\norm{C-C_e}_{\ell^2}
\leq
\frac{K_{\mathrm{stab}}}{2m_{\min}^{3/2}}
\norm{n}_{L^2(0,T;\mathbb{R}^{|B|})}.
$$
Since $m_{\min}=c_{\max}^{-2}$, this becomes
\begin{equation}
\label{eq:speed-recovery-stability-abstract}
\norm{C-C_e}_{\ell^2}
\leq
\frac{K_{\mathrm{stab}}}{2}c_{\max}^3
\norm{n}_{L^2(0,T;\mathbb{R}^{|B|})}.
\end{equation}

\subsection{Numerical reconstruction}
\label{section:algorithm}

We now describe the reconstruction method used in the numerical experiments of Section \ref{sec:numerics}. Although the uniqueness result above requires only one suitable initial signal, the numerical method allows several experiments, which can improve the conditioning of the reconstruction.

Let
$$
0=t_0<t_1<\cdots<t_K=T
$$
be the observation times. Let
$$
f^{(r)}\in\mathbb{R}^N,
\qquad
r=1,\ldots,R,
$$
be known initial signals, with corresponding measured data
$$
y_q^{(r)}\in\mathbb{R}^{|B|},
\qquad
q=0,\ldots,K.
$$
Assume that the admissible speeds satisfy
$$
0<c_{\min}\leq c_i\leq c_{\max}<\infty,
$$
and set
$$
m_{\min}=c_{\max}^{-2},
\qquad
m_{\max}=c_{\min}^{-2}.
$$

We introduce the logarithmic parameters
$$
\theta_i=\log m_i.
$$
The admissible parameter set is therefore
$$
\Theta=[\log m_{\min},\log m_{\max}]^N,
$$
and the corresponding diagonal coefficient matrix is
$$
M(\theta)=\diag(e^{\theta_1},\ldots,e^{\theta_N}).
$$
For each $\theta\in\Theta$ and each $r=1,\ldots,R$, let $X^{(r)}(\cdot;\theta)$ solve
$$
\begin{cases}
M(\theta)X_{tt}^{(r)}+LX^{(r)}=0 & \text{on }[0,T],\\
X^{(r)}(0)=f^{(r)},\quad X_t^{(r)}(0)=0.
\end{cases}
$$

Let $w_q$ denote the trapezoidal quadrature weights associated with the observation grid. Explicitly,
$$
w_0=\frac{t_1-t_0}{2},
\qquad
w_K=\frac{t_K-t_{K-1}}{2},
$$
and
$$
w_q=\frac{t_{q+1}-t_{q-1}}{2},
\qquad
q=1,\ldots,K-1.
$$
For a uniform grid $t_q=q\Delta t$, these reduce to
$$
w_0=w_K=\frac{\Delta t}{2},
\qquad
w_q=\Delta t,
\quad q=1,\ldots,K-1.
$$

We define the weighted residuals by
\begin{equation}
\label{eq:data-residual}
r_q^{(r)}(\theta)
=
\sqrt{w_q}
\left(
R_BX^{(r)}(t_q;\theta)-y_q^{(r)}
\right),
\end{equation}
and form $r(\theta)$ by stacking these vectors over $r$ and $q$. The factors $\sqrt{w_q}$ ensure that the squared Euclidean norm of the residual is the trapezoidal approximation of the continuous $L^2$ data misfit. The reconstruction problem is
\begin{equation}
\label{eq:discrete-least-squares}
\min_{\theta\in\Theta}
\frac{1}{2}\norm{r(\theta)}_2^2.
\end{equation}

To compute the Jacobian of the residual, define
$$
S_j^{(r)}
=
\frac{\partial X^{(r)}}{\partial\theta_j},
\qquad
S^{(r)}
=
\begin{pmatrix}
S_1^{(r)}&\cdots&S_N^{(r)}
\end{pmatrix}.
$$
Since
$$
\frac{\partial M(\theta)}{\partial\theta_j}
=
m_jE_{jj},
$$
where $E_{jj}$ has a single nonzero entry equal to one in position $(j,j)$, differentiating the forward equation with respect to $\theta_j$ gives
$$
M(\theta)(S_j^{(r)})_{tt}
+
LS_j^{(r)}
=
-m_jE_{jj}X_{tt}^{(r)}.
$$
The initial data are independent of $\theta$, so
$$
S_j^{(r)}(0)=0,
\qquad
(S_j^{(r)})_t(0)=0.
$$
Since
$$
M(\theta)^{-1}m_jE_{jj}=E_{jj},
$$
collecting the sensitivity vectors as columns gives
\begin{equation}
\label{eq:sensitivity-equation}
\begin{cases}
S_{tt}^{(r)}
=
-M(\theta)^{-1}LS^{(r)}
-
\diag\left(X_{tt}^{(r)}\right),\\
S^{(r)}(0)=0,\quad S_t^{(r)}(0)=0.
\end{cases}
\end{equation}
The corresponding blocks of the Jacobian are
$$
J_q^{(r)}(\theta)
=
\sqrt{w_q}\,R_BS^{(r)}(t_q;\theta).
$$

We solve \eqref{eq:discrete-least-squares} using a bounded trust-region reflective method. For background on trust-region methods for nonlinear least-squares problems, see \cite{NocedalWright}. The reconstruction algorithm is as follows.

\begin{enumerate}
    \item Choose an initial estimate $\theta^{(0)}\in\Theta$.

    \item For the current parameter vector $\theta$, solve the forward problem for each initial signal $f^{(r)}$.

    \item Solve the sensitivity equations \eqref{eq:sensitivity-equation} and assemble the residual and Jacobian using \eqref{eq:data-residual}.

    \item Use the residual and Jacobian to compute a bounded trust-region reflective update of $\theta$. Return to Step 2 until a prescribed stopping criterion is satisfied.

    \item If $\widehat{\theta}$ denotes the terminal iterate, recover
    $$
    \widehat{m}_i=e^{\widehat{\theta}_i},
    \qquad
    \widehat{c}_i=e^{-\widehat{\theta}_i/2}.
    $$
\end{enumerate}

The admissible set $\Theta$ is compact, and the forward solution depends continuously on $\theta$. Therefore the objective in \eqref{eq:discrete-least-squares} is continuous and attains a global minimum on $\Theta$. The objective is generally nonconvex, however, so the local optimization method need not reach a global minimizer, and its output may depend on the initialization.

\begin{rmk}
Suppose that the data are exact and generated by some $\theta^\dagger\in\Theta$. For each positive integer $K$, let
$$
\Delta t=\frac{T}{K},
\qquad
t_q=q\Delta t,
\quad q=0,\ldots,K,
$$
and let
$$
w_0=w_K=\frac{\Delta t}{2},
\qquad
w_q=\Delta t,
\quad q=1,\ldots,K-1.
$$
Define the discrete objective
$$
E_K(\theta)
=
\frac{1}{2}
\sum_{r=1}^R
\sum_{q=0}^K
w_q
\norm{
R_BX^{(r)}(t_q;\theta)
-
R_BX^{(r)}(t_q;\theta^\dagger)
}_2^2.
$$
As $K\to\infty$, the functions $E_K$ converge uniformly on $\Theta$ to
$$
E_T(\theta)
=
\frac{1}{2}
\sum_{r=1}^R
\int_0^T
\norm{
R_BX^{(r)}(t;\theta)
-
R_BX^{(r)}(t;\theta^\dagger)
}_2^2
\,dt.
$$
Indeed, the integrands depend jointly continuously on $(t,\theta)$ on the compact set $[0,T]\times\Theta$, and are therefore uniformly continuous there.

Suppose that $B$ is a zero forcing set and that at least one initial signal, say $f^{(r_0)}$, satisfies
$$
(Lf^{(r_0)})_i\neq0
\qquad
\text{for every }i\in V.
$$
If $E_T(\theta)=0$, then the nonnegative continuous integrand corresponding to $f^{(r_0)}$ has zero integral and therefore vanishes on $[0,T]$. Hence
$$
R_BX^{(r_0)}(t;\theta)
=
R_BX^{(r_0)}(t;\theta^\dagger)
\qquad
\text{for every }t\in[0,T].
$$
Theorem \ref{thm:one-experiment-non-harmonic} then gives
$$
M(\theta)=M(\theta^\dagger),
$$
and therefore
$$
\theta=\theta^\dagger.
$$
Thus $\theta^\dagger$ is the unique global minimizer of $E_T$. Consequently, if $\widehat{\theta}_K$ is any global minimizer of $E_K$, then
$$
\widehat{\theta}_K\longrightarrow\theta^\dagger
$$
as $K\to\infty$.

This conclusion concerns global minimizers of the discrete objectives, not necessarily the iterates produced by the local least-squares method. For a fixed $K$, injectivity of the continuous-time measurement map does not by itself imply injectivity of the sampled measurement map
$$
\theta
\longmapsto
\left(
R_BX^{(r)}(t_q;\theta)
\right)_{r,q}.
$$
Using several initial signals can improve the conditioning and identifiability of the sampled problem, but does not guarantee global convergence of the optimization method.
\end{rmk}

The purpose of the present method is to demonstrate the feasibility of the reconstruction and investigate its numerical behavior, rather than to optimize computational efficiency. Developing faster algorithms with stronger convergence guarantees is left for future work.

\section{Extension to positively weighted graphs}
\label{sec:weighted}

The results above extend directly to positively weighted graphs.

\begin{defn}
Let $\Gamma=(V,E)$ be a connected graph, and assign a weight $\omega_{ij}>0$ to each edge $\{i,j\}\in E$. The corresponding weighted graph Laplacian $L_\omega$ is defined by
$$
(L_\omega)_{ii}=\sum_{j\sim i}\omega_{ij},
\qquad
(L_\omega)_{ij}
=
\begin{cases}
-\omega_{ij}, & i\sim j,\\
0, & i\not\sim j.
\end{cases}
$$
The unweighted graph Laplacian corresponds to $\omega_{ij}=1$ for every edge.
\end{defn}

\subsection{Initial-signal recovery}

Suppose that $C$ and $L_\omega$ are known. The graph wave equation becomes
$$
X_{tt}=-C^2L_\omega X,
$$
and the first-order coefficient matrix is
$$
A_\omega=
\begin{pmatrix}
0&I_N\\
-C^2L_\omega&0
\end{pmatrix}.
$$
The observability and reconstruction results of Section \ref{sec:source-recovery} remain unchanged after replacing $L$ by $L_\omega$. In particular, zero forcing still implies observability. Indeed, at a forcing step $u$ forces $w$, and the corresponding weighted Laplacian relation reduces to
$$
0=(L_\omega p)_u=-\omega_{uw}p_w.
$$
Since $\omega_{uw}>0$, it follows that $p_w=0$, and the same zero forcing argument applies.

Thus the reconstruction formula and stability estimate remain valid with $A$ and $W_o$ replaced by their weighted counterparts $A_\omega$ and $W_{o,\omega}$.

\subsection{Vertex-speed recovery}

In the mass-form parametrization, the weighted equation is
$$
MX_{tt}+L_\omega X=0.
$$
The nonharmonicity condition becomes
$$
(L_\omega f)_i\neq0
\qquad
\text{for every }i\in V.
$$
As in the unweighted case, this condition holds for almost every $f\in\mathbb{R}^N$.

The uniqueness proof also carries over directly. At a forcing step $u$ forces $w$, the relation
$$
(LW)_u=-W_w
$$
is replaced by
$$
(L_\omega W)_u=-\omega_{uw}W_w.
$$
Since $\omega_{uw}>0$, vanishing at $u$ again implies $W_w=0$, and equality of the coefficients propagates along the zero forcing sequence. Hence Theorem \ref{thm:one-experiment-non-harmonic} remains valid with $L$ replaced by $L_\omega$.

The same observation applies to the injectivity proof for the Fr\'echet derivative. Consequently, the Lipschitz stability results remain valid for positively weighted graphs, with the stability constant corresponding to the weighted parameter-to-data map.

The numerical reconstruction method likewise extends by replacing $L$ with $L_\omega$ throughout.

\section{Numerical experiments}

    \label{sec:numerics}

We report numerical experiments for the reconstruction method of Section \ref{section:algorithm}. The implementation is publicly available; see the Code availability statement at the end of the paper. The experiments assess the reconstruction in practice and examine how observation geometry, time horizon, coefficient localization, and measurement noise affect numerical conditioning and convergence.

\subsection{Implementation of the speed-recovery algorithm}

    Let $p=\lfloor N/2\rfloor$ and $q=N-p$, so that the $N$ entries are divided into two groups of sizes $p$ and $q$. Using a graph-specific seed, the entries of the zero-sum vector
$$
g=
\left(
\underbrace{1,\ldots,1}_{p\text{ entries}},
\underbrace{-p/q,\ldots,-p/q}_{q\text{ entries}}
\right)
$$
were randomly permuted. The initial condition was then obtained by solving the singular Laplacian equation
    \begin{equation*}
        L\widetilde{f}=g
    \end{equation*}
    on the orthogonal complement of the constant functions. Equivalently, if
    \begin{equation*}
        L\phi_k=\lambda_k\phi_k, \qquad 0=\lambda_0<\lambda_1\leq\cdots,
    \end{equation*}
    then
    \begin{equation*}
        \widetilde{f} = \sum_{\lambda_k>10^{-12}\max\{1,\lambda_{\max}\}}
        \frac{\langle g,\phi_k\rangle}{\lambda_k}\phi_k.
    \end{equation*}
    The mean of $\widetilde{f}$ was removed and the result was normalized:
    \begin{equation*}
        f = \frac{\widetilde{f} -N^{-1}\langle\widetilde{f},\mathbf{1}\rangle\mathbf{1}}{\norm{\widetilde{f} -N^{-1}\langle\widetilde{f},\mathbf{1}\rangle\mathbf{1}}_2}.
    \end{equation*}
    Applying the inverse Laplacian suppresses high-frequency components, so the resulting initial condition has reduced high-frequency variation across the graph. At the same time, $Lf$ is proportional to the permuted vector $g$ and therefore has no zero components. The implementation verified the quantitative nonharmonicity condition
    \begin{equation*}
        \frac{\min_{i\in V}|(Lf) \vert_i|} {\norm{Lf}_2/\sqrt{N}} \geq 0.5.
    \end{equation*}
    For each reference graph, the same initial condition was used across all observation sets, time horizons, and treatments of the speeds on the observation set. 
    
    In the numerical trials, the true speeds belonged to $[0.75,1.25]$, while the optimization was carried out over the larger interval
    \begin{equation*}
        0.5625 \leq c_i \leq 1.5625.
    \end{equation*}
    The corresponding logarithmic bounds were imposed componentwise. When the speeds on $B$ were assumed known, their masses were fixed at their prescribed values and only the parameters on $V\setminus B$ were optimized. Otherwise, all vertex masses were treated as unknown.
    
    No regularization term or robust loss was used in the trials reported here. Consequently, the residual supplied to the optimizer was precisely \eqref{eq:data-residual} over all initial conditions, observation times, and observed vertices.
    
    The state and its parameter sensitivities were integrated simultaneously. The sensitivity equations used by the implementation were those in \eqref{eq:sensitivity-equation}. If the speeds on $B$ were fixed, only the columns corresponding to $V\setminus B$ were passed to the optimizer.
    
    The initial mass estimate was obtained from the short-time curvature of the measured traces. At each observed vertex $b$, the first eight temporal samples were fitted by an even polynomial
    \begin{equation*}
        X^{(r)}(t) \vert_b -f_b^{(r)}
        \approx
        \sum_{k=1}^{3} a_{b,k}^{(r)}t^{2k}.
    \end{equation*}
    Since
    \begin{equation*}
        m_b{X}_{tt}^{(r)}(0) \vert_b = -(Lf^{(r)}) \vert_b,
    \end{equation*}
    the mass at the observed vertex was estimated by
    \begin{equation*}
        m_b^{(0)} = -\frac{(Lf^{(r)}) \vert_b}{2a_{b,1}^{(r)}}.
    \end{equation*}
    Estimates with nearly zero fitted acceleration, poor conditioning, or values outside the admissible interval were discarded. The remaining estimates were combined by their median. Uninitialized vertices were assigned the median reliable observed-vertex estimate, or the geometric midpoint of the mass interval if no reliable estimate was available.
    
    The bounded nonlinear least-squares problem was solved using the trust-region reflective algorithm in \texttt{scipy.optimize.least\_squares}, with Jacobian scaling and the analytic sensitivity matrix described above. The forward and sensitivity systems were integrated in double precision using the eighth-order DOP853 method, with relative and absolute tolerances $10^{-10}$ and $10^{-12}$, respectively. The three least-squares stopping tolerances were set to $10^{-12}$, and the maximum number of forward evaluations was
    \begin{equation*}
        \max\{50000,500N\}.
    \end{equation*}
    
    For the synthetic experiments, the measured traces were generated independently by a generalized modal decomposition of the pair $(L,M)$. The inverse solver nevertheless evaluated its forward map and Jacobian by numerical integration of the state and sensitivity equations. Conditioning was assessed from the singular values of the trapezoidally weighted Jacobian restricted to the free parameters. An effective rank was computed using the relative threshold $10^{-10}$.

\subsection{Small-graph experiments}

    We first tested a collection of fixed graph geometries with $N\leq30$, taking measurements on a minimal zero forcing set computed separately from the graph geometry. The results are reported in Table \ref{tab:low-n-experiments}. In this and all subsequent tables, the truth condition number is the ratio of the largest to smallest singular values of the boundary-data Jacobian evaluated at the ground truth parameters. It quantifies identifiability, with larger values indicating greater sensitivity to perturbations of the observed data. 

    \begin{table}[h!tbp]
    \centering
    \caption{Exact-data reconstruction results for low-$N$ graphs.}
    \label{tab:low-n-experiments}
    \resizebox{\textwidth}{!}{%
    \begin{tabular}{lrrrrrrl}
        \toprule
        Graph
        & $N$
        & \shortstack{Observed\\vertices}
        & \shortstack{Relative\\speed error}
        & \shortstack{Trace\\residual}
        & \shortstack{Truth\\condition}
        & \shortstack{Inversion\\seconds}
        & Recovery \\
        \midrule
        Path & 10 & 1 & $5.56\times10^{-1}$  & $3.91\times10^{-1}$  & 60.06 & 2.22 & Failure \\
        Path & 20 & 1 & $6.42\times10^{-12}$ & $1.33\times10^{-12}$ & 90.63 & 1.77 & Success \\
        \addlinespace
        Cycle & 10 & 2 & $6.93\times10^{-12}$ & $4.40\times10^{-12}$ & 21.32 & 0.41 & Success \\
        Cycle & 20 & 2 & $1.56\times10^{-1}$  & $3.24\times10^{-2}$  & 40.60 & 8.84 & Failure \\
        \addlinespace
        Ladder & 10 & 2 & $2.96\times10^{-1}$  & $1.29\times10^{-1}$  & 8.86  & 1.87 & Failure \\
        Ladder & 20 & 2 & $5.51\times10^{-1}$  & $4.07\times10^{-1}$  & 23.99 & 6.74 & Failure \\
        \addlinespace
        Comb tree & 10 & 5  & $1.41\times10^{-12}$ & $2.18\times10^{-11}$ & 6.22  & 0.34 & Success \\
        Comb tree & 25 & 12 & $8.27\times10^{-12}$ & $3.79\times10^{-11}$ & 10.22 & 0.82 & Success \\
        \addlinespace
        Balanced tree & 15 & 5 & $3.66\times10^{-12}$ & $8.82\times10^{-12}$ & 9.61 & 0.32 & Success \\
        \addlinespace
        Square grid & 16 & 4 & $2.10\times10^{-12}$ & $1.29\times10^{-11}$ & 8.18 & 0.50 & Success \\
        \addlinespace
        Cubic grid & 27 & 9 & $8.82\times10^{-12}$ & $6.74\times10^{-11}$ & 5.25 & 2.29 & Success \\
        \addlinespace
        Erd\H{o}s--R\'enyi $(\bar d=4,\ s=0)$ & 10 & 2 & $2.37\times10^{-12}$ & $2.66\times10^{-12}$ & 13.86 & 0.36 & Success \\
        Erd\H{o}s--R\'enyi $(\bar d=4,\ s=1)$ & 10 & 4 & $3.17\times10^{-12}$ & $2.12\times10^{-11}$ & 7.00  & 0.18 & Success \\
        Erd\H{o}s--R\'enyi $(\bar d=8,\ s=0)$ & 10 & 7 & $6.00\times10^{-12}$ & $6.93\times10^{-11}$ & 3.12  & 0.17 & Success \\
        Erd\H{o}s--R\'enyi $(\bar d=8,\ s=1)$ & 10 & 8 & $1.24\times10^{-1}$  & $3.78\times10^{-1}$  & 1.22  & 0.46 & Failure \\
        Erd\H{o}s--R\'enyi $(\bar d=4,\ s=0)$ & 25 & 9  & $6.89\times10^{-12}$ & $2.90\times10^{-11}$ & 5.80 & 0.33 & Success \\
        Erd\H{o}s--R\'enyi $(\bar d=4,\ s=1)$ & 25 & 8  & $1.46\times10^{-11}$ & $1.68\times10^{-11}$ & 9.45 & 0.37 & Success \\
        Erd\H{o}s--R\'enyi $(\bar d=8,\ s=0)$ & 25 & 14 & $3.07\times10^{-12}$ & $4.91\times10^{-11}$ & 4.84 & 0.36 & Success \\
        Erd\H{o}s--R\'enyi $(\bar d=8,\ s=1)$ & 25 & 14 & $4.58\times10^{-12}$ & $3.20\times10^{-11}$ & 3.73 & 0.35 & Success \\
        \bottomrule
    \end{tabular}%
    }
\end{table}
    \FloatBarrier

    Overall, for small graphs, the algorithm performs quite well. The notable failures are the $10$-node path (on which the minimal zero forcing set is a single end node), the two ``ladder'' type graphs (on which the minimal zero forcing set is the bottom rung), and the $20$-node cycle (on which the minimal zero forcing set is a pair of adjacent nodes), and one of the Erd\H{o}s--Rényi graphs with $10$ nodes. The residuals and relative speed errors for all other trials are of order $-11$ or $-12$, indicating machine precision. The trials that failed suggest that the graph geometry---and thereby the minimal zero forcing set used in the algorithm---exercise the greatest influence with respect to its chances for success.

\subsection{Observation-set and time-horizon experiments}

    The uniqueness theorem does not imply that a minimal zero forcing set is numerically optimal, nor does it guarantee that the local least-squares solver reaches the global minimizer. We therefore compared several observation-set geometries on two-dimensional square and three-dimensional cubic lattice graphs. A single side of the square is a minimal zero forcing set for the square grids considered here, while a single face is minimal for the cubic grids.

    For each square grid, the locally nonharmonic initial condition and the speeds were generated from the fixed graph-specific seed. The inverse algorithm was then run with the observation set consisting of a single side, two perpendicular sides, three sides (the third parallel to the first), and the entire perimeter. The experiments were repeated with widened observation horizons (the parameter corresponding to $T$) with multipliers of $1.5$, $2$, and $4$. These results are recorded in Table \ref{tab:square-progressive-boundary}.

    \begin{table}[h!tbp]
    \centering
    \caption{Progressive observation-set reconstruction results on the square grids.}
    \label{tab:square-progressive-boundary}
    \resizebox{\textwidth}{!}{%
    \begin{tabular}{lrrrrrrl}
        \toprule
        Dimension
        & \shortstack{Square sides\\observed}
        & \shortstack{Horizon\\multiplier}
        & \shortstack{Relative\\speed error}
        & \shortstack{Trace\\residual}
        & \shortstack{Truth\\condition}
        & \shortstack{Inversion\\seconds}
        & Recovery \\
        \midrule
        $10\times 10$ & 1 & 1.0 & $3.71\times 10^{-1}$  & $8.40\times 10^{-2}$  & 116.85 & 562.16   & Failure \\
        $10\times 10$ & 1 & 1.5 & $4.53\times 10^{-1}$  & $7.99\times 10^{-1}$  & 122.39 & 46293.94 & Failure \\
        $10\times 10$ & 1 & 2.0 & $4.57\times 10^{-1}$  & $8.39\times 10^{-1}$  & 149.46 & 1362.95  & Failure \\
        $10\times 10$ & 1 & 4.0 & $4.17\times 10^{-1}$  & $8.00\times 10^{-1}$  & 185.59 & 6772.75  & Failure \\
        \midrule
        $10\times 10$ & 2 & 1.0 & $3.95\times 10^{-2}$  & $8.63\times 10^{-3}$  & 51.15  & 137.75   & Failure \\
        $10\times 10$ & 2 & 1.5 & $3.49\times 10^{-1}$  & $1.25\times 10^{-1}$  & 50.96  & 1784.67  & Failure \\
        $10\times 10$ & 2 & 2.0 & $4.22\times 10^{-1}$  & $9.27\times 10^{-1}$  & 52.37  & 792.92   & Failure \\
        $10\times 10$ & 2 & 4.0 & $3.64\times 10^{-1}$  & $8.79\times 10^{-1}$  & 79.53  & 5724.44  & Failure \\
        \midrule
        $10\times 10$ & 3 & 1.0 & $5.89\times 10^{-11}$ & $5.16\times 10^{-11}$ & 26.50  & 21.92    & Success \\
        $10\times 10$ & 3 & 1.5 & $5.76\times 10^{-11}$ & $5.96\times 10^{-11}$ & 32.10  & 48.25    & Success \\
        $10\times 10$ & 3 & 2.0 & $5.87\times 10^{-11}$ & $6.28\times 10^{-11}$ & 34.76  & 86.99    & Success \\
        $10\times 10$ & 3 & 4.0 & $1.03\times 10^{-1}$  & $8.10\times 10^{-2}$  & 51.83  & 2015.18  & Failure \\
        \midrule
        $10\times 10$ & 4 & 1.0 & $5.32\times 10^{-11}$ & $5.44\times 10^{-11}$ & 19.97  & 15.22    & Success \\
        $10\times 10$ & 4 & 1.5 & $5.86\times 10^{-11}$ & $6.43\times 10^{-11}$ & 24.17  & 34.74    & Success \\
        $10\times 10$ & 4 & 2.0 & $6.24\times 10^{-11}$ & $6.94\times 10^{-11}$ & 26.90  & 58.43    & Success \\
        $10\times 10$ & 4 & 4.0 & $8.46\times 10^{-2}$  & $8.46\times 10^{-2}$  & 39.48  & 4902.14  & Failure \\
        \midrule
        $12\times 12$ & 1 & 1.0 & $1.05\times 10^{-1}$  & $4.22\times 10^{-2}$  & 37.11 & 2087.84  & Failure \\
        $12\times 12$ & 1 & 1.5 & $1.44\times 10^{-1}$  & $5.69\times 10^{-2}$  & 37.10 & 1558.95  & Failure \\
        $12\times 12$ & 1 & 2.0 & $2.53\times 10^{-1}$  & $1.26\times 10^{-1}$  & 39.47 & 5388.30  & Failure \\
        $12\times 12$ & 1 & 4.0 & $3.62\times 10^{-1}$  & $7.97\times 10^{-1}$  & 60.27 & 9710.31  & Failure \\
        \midrule
        $12\times 12$ & 2 & 1.0 & $2.71\times 10^{-11}$ & $3.27\times 10^{-11}$ & 29.61 & 97.90    & Success \\
        $12\times 12$ & 2 & 1.5 & $1.62\times 10^{-1}$  & $7.26\times 10^{-2}$  & 30.87 & 1697.34  & Failure \\
        $12\times 12$ & 2 & 2.0 & $2.58\times 10^{-11}$ & $4.56\times 10^{-11}$ & 33.84 & 1462.93  & Success \\
        $12\times 12$ & 2 & 4.0 & $3.84\times 10^{-1}$  & $7.41\times 10^{-1}$  & 49.07 & 19029.97 & Failure \\
        \midrule
        $12\times 12$ & 3 & 1.0 & $2.42\times 10^{-11}$ & $5.16\times 10^{-11}$ & 18.80 & 52.18    & Success \\
        $12\times 12$ & 3 & 1.5 & $2.52\times 10^{-11}$ & $6.31\times 10^{-11}$ & 21.43 & 220.57   & Success \\
        $12\times 12$ & 3 & 2.0 & $1.48\times 10^{-1}$  & $1.38\times 10^{-1}$  & 24.60 & 6386.34  & Failure \\
        $12\times 12$ & 3 & 4.0 & $3.29\times 10^{-1}$  & $7.98\times 10^{-1}$  & 36.89 & 31753.44 & Failure \\
        \midrule
        $12\times 12$ & 4 & 1.0 & $3.28\times 10^{-11}$ & $8.99\times 10^{-11}$ & 14.22 & 52.73    & Success \\
        $12\times 12$ & 4 & 1.5 & $3.25\times 10^{-11}$ & $1.11\times 10^{-10}$ & 16.27 & 126.38   & Success \\
        $12\times 12$ & 4 & 2.0 & $3.39\times 10^{-11}$ & $1.30\times 10^{-10}$ & 18.00 & 250.75   & Success \\
        $12\times 12$ & 4 & 4.0 & $9.42\times 10^{-2}$  & $1.57\times 10^{-1}$  & 27.37 & 5707.85  & Failure \\
        \bottomrule
    \end{tabular}%
    }
\end{table}
    \FloatBarrier

    For each cubic grid, the locally nonharmonic initial condition and the speeds were generated from the fixed graph-specific seed. The inverse algorithm was then run with the observation set consisting of a single face, two perpendicular faces, three perpendicular faces, four faces (the fourth parallel to the first), five faces (the fifth parallel to the second), and the entire exterior shell. The experiment was repeated with observation horizon widened by a multiplier of $1.5$. These results are recorded in Table \ref{tab:cube-progressive-boundary}.
    
    \begin{table}[h!tbp]
    \centering
    \caption{Progressive observation-set reconstruction results on the cubic grids.}
    \label{tab:cube-progressive-boundary}
    \resizebox{\textwidth}{!}{%
    \begin{tabular}{lrrrrrrl}
        \toprule
        Dimension
        & \shortstack{Faces\\observed}
        & \shortstack{Horizon\\multiplier}
        & \shortstack{Relative\\speed error}
        & \shortstack{Trace\\residual}
        & \shortstack{Truth\\condition}
        & \shortstack{Inversion\\seconds}
        & Recovery \\
        \midrule
        $5\times5\times5$ & 1 & 1.0 & $2.09\times10^{-1}$  & $2.39\times10^{-1}$  & 12.36 & 436.56  & Failure \\
        $5\times5\times5$ & 1 & 1.5 & $2.90\times10^{-1}$  & $7.89\times10^{-1}$  & 12.79 & 1092.49 & Failure \\
        \midrule
        $5\times5\times5$ & 2 & 1.0 & $9.21\times10^{-2}$  & $2.05\times10^{-1}$  & 8.98  & 183.16  & Failure \\
        $5\times5\times5$ & 2 & 1.5 & $2.73\times10^{-1}$  & $7.98\times10^{-1}$  & 9.15  & 917.67  & Failure \\
        \midrule
        $5\times5\times5$ & 3 & 1.0 & $1.72\times10^{-11}$ & $9.02\times10^{-11}$ & 7.20  & 19.87   & Success \\
        $5\times5\times5$ & 3 & 1.5 & $1.74\times10^{-11}$ & $1.16\times10^{-10}$ & 6.99  & 45.67   & Success \\
        \midrule
        $5\times5\times5$ & 4 & 1.0 & $1.44\times10^{-11}$ & $1.07\times10^{-10}$ & 5.37  & 17.06   & Success \\
        $5\times5\times5$ & 4 & 1.5 & $1.65\times10^{-11}$ & $1.31\times10^{-10}$ & 5.35  & 32.49   & Success \\
        \midrule
        $5\times5\times5$ & 5 & 1.0 & $1.33\times10^{-11}$ & $1.12\times10^{-10}$ & 3.98  & 15.43   & Success \\
        $5\times5\times5$ & 5 & 1.5 & $1.50\times10^{-11}$ & $1.41\times10^{-10}$ & 4.11  & 27.61   & Success \\
        \midrule
        $5\times5\times5$ & 6 & 1.0 & $1.53\times10^{-11}$ & $1.37\times10^{-10}$ & 3.17  & 10.36   & Success \\
        $5\times5\times5$ & 6 & 1.5 & $1.66\times10^{-11}$ & $1.71\times10^{-10}$ & 3.13  & 18.12   & Success \\
        \midrule
        $6\times6\times6$ & 1 & 1.0 & $4.10\times10^{-1}$  & $1.24\times10^{0}$   & 21.51 & 4549.71  & Failure \\
        $6\times6\times6$ & 1 & 1.5 & $4.13\times10^{-1}$  & $9.61\times10^{-1}$  & 22.63 & 11483.04 & Failure \\
        \midrule
        $6\times6\times6$ & 2 & 1.0 & $1.72\times10^{-1}$  & $1.46\times10^{-1}$  & 14.70 & 7431.01  & Failure \\
        $6\times6\times6$ & 2 & 1.5 & $2.01\times10^{-1}$  & $1.92\times10^{-1}$  & 15.14 & 15619.02 & Failure \\
        \midrule
        $6\times6\times6$ & 3 & 1.0 & $1.31\times10^{-1}$  & $1.48\times10^{-1}$  & 10.39 & 2734.65  & Failure \\
        $6\times6\times6$ & 3 & 1.5 & $2.64\times10^{-1}$  & $5.27\times10^{-1}$  & 11.06 & 7337.49  & Failure \\
        \midrule
        $6\times6\times6$ & 4 & 1.0 & $4.18\times10^{-11}$ & $1.64\times10^{-10}$ & 6.91  & 84.90    & Success \\
        $6\times6\times6$ & 4 & 1.5 & $4.28\times10^{-11}$ & $1.98\times10^{-10}$ & 7.89  & 152.60   & Success \\
        \midrule
        $6\times6\times6$ & 5 & 1.0 & $3.70\times10^{-11}$ & $1.69\times10^{-10}$ & 5.54  & 79.27    & Success \\
        $6\times6\times6$ & 5 & 1.5 & $3.83\times10^{-11}$ & $2.05\times10^{-10}$ & 6.56  & 145.97   & Success \\
        \midrule
        $6\times6\times6$ & 6 & 1.0 & $3.37\times10^{-11}$ & $1.68\times10^{-10}$ & 4.20  & 77.84    & Success \\
        $6\times6\times6$ & 6 & 1.5 & $3.44\times10^{-11}$ & $2.07\times10^{-10}$ & 4.97  & 132.58   & Success \\
        \bottomrule
    \end{tabular}%
    }
\end{table}
    \FloatBarrier

    In both the square and cubic experiments, additional observation vertices generally improve reconstruction. The results also show a sharp separation between failed runs, with relative errors typically of order $10^{-1}$, and successful runs, which reach errors near $10^{-10}$ or smaller.

    Longer observation horizons also increase the computational cost substantially. In the square-grid experiments, apart from the two-side $12\times12$ case, increasing the horizon tends to reduce the success rate. A plausible explanation is phase accumulation: small frequency errors produce increasing phase mismatch over time, which can introduce additional local minima in the least-squares objective. We did not directly track the evolution of the objective landscape with the horizon, so this explanation should be regarded as heuristic.

\subsection{Localized-speed experiments}

    The cubic progressive observation-set experiments were also repeated on a localized wave-speed field. The localized fields were generated by adding an exponentially decaying bump function to the standard values. The bump was centered at a vertex maximally distant from the observation set and with decay radius $\max\{1, 0.12 \diam(\Gamma)\}$, mixed with the standard speeds at weight $0.2$, and rescaled to the prescribed interval $[0.75, 1.25]$. The result is that most of the relative differences in speed occur in a small region of the graph near the maximally distant vertex. The intent was to create a localized-contrast test case, motivated by imaging settings in which one seeks to identify a spatially localized anomaly. These results are recorded in Table \ref{tab:cube-localized-progressive-boundary}.

    \begin{table}[h!tbp]
    \centering
    \caption{Progressive observation-set reconstruction results for localized speed terms on the cubic grids.}
    \label{tab:cube-localized-progressive-boundary}
    \resizebox{\textwidth}{!}{%
    \begin{tabular}{lrrrrrrl}
        \toprule
        Dimension
        & \shortstack{Faces\\observed}
        & \shortstack{Horizon\\multiplier}
        & \shortstack{Relative\\speed error}
        & \shortstack{Trace\\residual}
        & \shortstack{Truth\\condition}
        & \shortstack{Inversion\\seconds}
        & Recovery \\
        \midrule
        $5\times5\times5$ & 1 & 1.0 & $1.70\times10^{-11}$ & $5.42\times10^{-11}$ & 14.31 & 13.17 & Success \\
        $5\times5\times5$ & 1 & 1.5 & $1.53\times10^{-11}$ & $7.30\times10^{-11}$ & 15.43 & 26.55 & Success \\
        \midrule
        $5\times5\times5$ & 2 & 1.0 & $1.76\times10^{-11}$ & $8.60\times10^{-11}$ & 7.75  & 9.76  & Success \\
        $5\times5\times5$ & 2 & 1.5 & $1.65\times10^{-11}$ & $1.13\times10^{-10}$ & 7.95  & 13.20 & Success \\
        \midrule
        $5\times5\times5$ & 3 & 1.0 & $1.57\times10^{-11}$ & $9.54\times10^{-11}$ & 5.97  & 7.66  & Success \\
        $5\times5\times5$ & 3 & 1.5 & $1.66\times10^{-11}$ & $1.23\times10^{-10}$ & 6.27  & 11.71 & Success \\
        \midrule
        $5\times5\times5$ & 4 & 1.0 & $1.10\times10^{-11}$ & $1.07\times10^{-10}$ & 4.30  & 5.96  & Success \\
        $5\times5\times5$ & 4 & 1.5 & $1.33\times10^{-11}$ & $1.35\times10^{-10}$ & 4.78  & 10.04 & Success \\
        \midrule
        $5\times5\times5$ & 5 & 1.0 & $1.14\times10^{-11}$ & $1.08\times10^{-10}$ & 3.77  & 5.50  & Success \\
        $5\times5\times5$ & 5 & 1.5 & $1.24\times10^{-11}$ & $1.37\times10^{-10}$ & 4.20  & 8.27  & Success \\
        \midrule
        $5\times5\times5$ & 6 & 1.0 & $1.11\times10^{-11}$ & $1.12\times10^{-10}$ & 3.44  & 5.34  & Success \\
        $5\times5\times5$ & 6 & 1.5 & $1.23\times10^{-11}$ & $1.41\times10^{-10}$ & 3.68  & 8.02  & Success \\
        \midrule
        $6\times6\times6$ & 1 & 1.0 & $3.97\times10^{-11}$ & $7.19\times10^{-11}$ & 20.79 & 76.27  & Success \\
        $6\times6\times6$ & 1 & 1.5 & $3.47\times10^{-11}$ & $9.90\times10^{-11}$ & 20.20 & 144.03 & Success \\
        \midrule
        $6\times6\times6$ & 2 & 1.0 & $4.70\times10^{-11}$ & $1.47\times10^{-10}$ & 13.00 & 54.02  & Success \\
        $6\times6\times6$ & 2 & 1.5 & $3.89\times10^{-11}$ & $1.70\times10^{-10}$ & 13.08 & 91.82  & Success \\
        \midrule
        $6\times6\times6$ & 3 & 1.0 & $4.09\times10^{-11}$ & $1.39\times10^{-10}$ & 10.81 & 37.87  & Success \\
        $6\times6\times6$ & 3 & 1.5 & $3.70\times10^{-11}$ & $1.65\times10^{-10}$ & 10.88 & 71.62  & Success \\
        \midrule
        $6\times6\times6$ & 4 & 1.0 & $3.60\times10^{-11}$ & $1.43\times10^{-10}$ & 8.25  & 37.74  & Success \\
        $6\times6\times6$ & 4 & 1.5 & $3.34\times10^{-11}$ & $1.69\times10^{-10}$ & 9.05  & 61.14  & Success \\
        \midrule
        $6\times6\times6$ & 5 & 1.0 & $3.30\times10^{-11}$ & $1.46\times10^{-10}$ & 6.51  & 37.21  & Success \\
        $6\times6\times6$ & 5 & 1.5 & $3.17\times10^{-11}$ & $1.74\times10^{-10}$ & 7.04  & 58.26  & Success \\
        \midrule
        $6\times6\times6$ & 6 & 1.0 & $2.86\times10^{-11}$ & $1.47\times10^{-10}$ & 4.85  & 37.79  & Success \\
        $6\times6\times6$ & 6 & 1.5 & $2.71\times10^{-11}$ & $1.80\times10^{-10}$ & 5.14  & 59.32  & Success \\
        \bottomrule
    \end{tabular}%
    }
\end{table}
    \FloatBarrier

    The localized-speed experiments are markedly easier: every trial in Table \ref{tab:cube-localized-progressive-boundary} converges to machine precision, and the inversion times are substantially lower on the same hardware. Among the successful cases, localization reduces the average computation time by more than one half. The most expensive localized trial is the single-face $6\times6\times6$ experiment with horizon multiplier $1.5$.

    \subsection{Noisy-data experiments}

    We repeated the successful cubic progressive-boundary experiments with noisy measurements to examine the effect of measurement error on the reconstruction. We restrict these tests to configurations for which the noiseless reconstruction converges, so that the effect of measurement noise can be separated from failures of the local optimization method. The four-, five-, and six-face observation sets were tested at relative noise levels of $0.1\%$, $0.5\%$, and $1.0\%$ for both the localized and nonlocalized speed fields. The results are recorded in Table \ref{tab:noisy-cube-progressive-boundary}.

    The noisy measurements were generated by first drawing a raw perturbation with independent entries from $\mathcal{N}(0,1)$. For each trial, this perturbation was generated using a fixed deterministic seed and then rescaled so that
    $$
    \frac{\norm{n}_{L^2(0,T;\mathbb{R}^{|B|})}}{\norm{R_BX}_{L^2(0,T;\mathbb{R}^{|B|})}}
    =\delta,
    $$
    where $\delta\in\{0.001,0.005,0.01\}$. The normalization was applied jointly over all observed vertices and time samples rather than separately at each vertex. Within each trial, the same perturbation direction was used for all three noise levels, with only its magnitude rescaled.

    \begin{table}[h!tbp]
    \centering
    \caption{Noisy-data reconstruction results on the $6\times6\times6$ cubic grid}
    \label{tab:noisy-cube-progressive-boundary}
    \resizebox{\textwidth}{!}{%
    \begin{tabular}{lrrrrrrr}
        \toprule
        \shortstack{Speed\\profile}
        & \shortstack{Faces\\observed}
        & \shortstack{Relative\\noise level}
        & \shortstack{Relative\\speed error}
        & \shortstack{Trace\\residual}
        & \shortstack{Error\\amplification}
        & \shortstack{Truth\\condition}
        & \shortstack{Inversion\\seconds} \\
        \midrule
        Smooth & 4 & 0.1\% & $7.11\times10^{-6}$ & $1.00\times10^{-3}$ & $7.11\times10^{-3}$ & 6.91 & 108.83 \\
        Smooth & 4 & 0.5\% & $3.56\times10^{-5}$ & $5.00\times10^{-3}$ & $7.12\times10^{-3}$ & 6.91 & 287.05 \\
        Smooth & 4 & 1.0\% & $7.12\times10^{-5}$ & $1.00\times10^{-2}$ & $7.12\times10^{-3}$ & 6.91 & 109.47 \\
        Smooth & 5 & 0.1\% & $6.81\times10^{-6}$ & $1.00\times10^{-3}$ & $6.81\times10^{-3}$ & 5.54 & 107.37 \\
        Smooth & 5 & 0.5\% & $3.41\times10^{-5}$ & $5.00\times10^{-3}$ & $6.81\times10^{-3}$ & 5.54 & 293.75 \\
        Smooth & 5 & 1.0\% & $6.81\times10^{-5}$ & $1.00\times10^{-2}$ & $6.81\times10^{-3}$ & 5.54 & 841.07 \\
        Smooth & 6 & 0.1\% & $4.84\times10^{-6}$ & $1.00\times10^{-3}$ & $4.84\times10^{-3}$ & 4.20 & 113.72 \\
        Smooth & 6 & 0.5\% & $2.42\times10^{-5}$ & $5.00\times10^{-3}$ & $4.84\times10^{-3}$ & 4.20 & 260.23 \\
        Smooth & 6 & 1.0\% & $4.84\times10^{-5}$ & $1.00\times10^{-2}$ & $4.84\times10^{-3}$ & 4.20 & 329.67 \\
        \midrule
        Localized & 4 & 0.1\% & $7.62\times10^{-6}$ & $1.00\times10^{-3}$ & $7.62\times10^{-3}$ & 8.25 & 52.10 \\
        Localized & 4 & 0.5\% & $7.92\times10^{-2}$ & $6.81\times10^{-2}$ & $1.58\times10^{1}$  & 8.25 & 660.77 \\
        Localized & 4 & 1.0\% & $7.62\times10^{-5}$ & $1.00\times10^{-2}$ & $7.62\times10^{-3}$ & 8.25 & 1833.13 \\
        Localized & 5 & 0.1\% & $6.58\times10^{-6}$ & $1.00\times10^{-3}$ & $6.58\times10^{-3}$ & 6.51 & 55.75 \\
        Localized & 5 & 0.5\% & $3.29\times10^{-5}$ & $5.00\times10^{-3}$ & $6.58\times10^{-3}$ & 6.51 & 345.66 \\
        Localized & 5 & 1.0\% & $6.58\times10^{-5}$ & $1.00\times10^{-2}$ & $6.58\times10^{-3}$ & 6.51 & 568.70 \\
        Localized & 6 & 0.1\% & $5.74\times10^{-6}$ & $1.00\times10^{-3}$ & $5.74\times10^{-3}$ & 4.85 & 45.41 \\
        Localized & 6 & 0.5\% & $2.87\times10^{-5}$ & $5.00\times10^{-3}$ & $5.74\times10^{-3}$ & 4.85 & 216.65 \\
        Localized & 6 & 1.0\% & $5.74\times10^{-5}$ & $1.00\times10^{-2}$ & $5.74\times10^{-3}$ & 4.85 & 355.29 \\
        \bottomrule
    \end{tabular}%
    }
\end{table}

     \FloatBarrier

    Apart from one anomalous trial, the reconstruction error scales approximately linearly with the noise level. In particular, increasing the noise from $0.1\%$ to $1.0\%$ increases the reconstruction error by roughly one order of magnitude. For a fixed observation set, the error amplification is nearly independent of the noise level and generally decreases as more vertices are observed. The trace residual is also comparable to the prescribed noise level, indicating that the reconstruction fits the data to approximately the accuracy justified by the noise. This behavior is consistent with the Lipschitz stability estimate above, which bounds the reconstruction error by the size of the measurement noise together with the data residual. These experiments are intended as a robustness check rather than a statistical study.

    Taken together, these computations complement the exact identifiability theory by separating structural uniqueness from practical conditioning. Zero forcing, together with the nonharmonicity condition for coefficient recovery, guarantees uniqueness from continuous-time data on any observation interval of positive length, but the observation geometry and time horizon can still have a substantial effect on the conditioning and optimization landscape of the reconstruction. The noisy-data experiments further indicate a roughly linear dependence of the reconstruction error on the measurement noise, consistent with the Lipschitz stability estimate. These observations also highlight the significance of the uniform-in-$h$ inverse stability issue identified in the introduction when passing to spatially refined discretizations.

\FloatBarrier
\section*{Acknowledgments}

\subsection*{Funding}
Amir Moradifam is supported by NSF grant DMS-1953620. 

\subsection*{Conflict of interest}
The authors declare no conflicts of interest.

\subsection*{Use of artificial intelligence}

    OpenAI's ChatGPT and Codex were used during the preparation of this work to assist with the numerical implementation, editorial refinement of the exposition, and identification of potentially relevant literature. The mathematical arguments and proofs were written by the authors, who also verified the statements of results, citations, and reported computations. The authors assume responsibility for all content.

\subsection*{Code availability}

    The Python implementation of the reconstruction algorithm was developed with substantial assistance from OpenAI's ChatGPT and Codex. ChatGPT was used to discuss implementation strategies and diagnose numerical procedures, while Codex was used to generate, modify, debug, and test portions of the source code. The authors inspected the code and checked it against the mathematical specification of the algorithm. The implementation and files used for the reported experiments are available at \url{https://github.com/Oliver-w8/Graph-Wave-Scientific-Solver}.

\printbibliography
\end{document}